\documentclass{amsart}
\usepackage{amsmath,amssymb}
\usepackage[dvips]{graphics}
\usepackage[dvipdfmx]{graphicx}
\usepackage[all]{xy}

\newtheorem{Theorem}{Theorem}[section]

\newtheorem{Proposition}[Theorem]{Proposition}

\theoremstyle{definition}

\newtheorem{Problem}[Theorem]{Problem}

\theoremstyle{remark}
\newtheorem{Remark}[Theorem]{Remark}

\makeatletter
\@addtoreset{figure}{section}%{subsection}
\def\@thmcountersep{-}
\makeatother

\numberwithin{equation}{section}

\begin{document} 

\title[Intrinsic linking of a simplicial $n$-complex embedded in $\mathbb{R}^{2n}$]{Intrinsic linking of a simplicial $n$-complex embedded in $\mathbb{R}^{2n}$}

\author{Ryo Nikkuni}
\address{Department of Information and Mathematical Sciences, School of Arts and Sciences, Tokyo Woman's Christian University, 2-6-1 Zempukuji, Suginami-ku, Tokyo 167-8585, Japan}
\email{nick@lab.twcu.ac.jp}
%\thanks{The author was supported by JSPS KAKENHI Grant Number 25K06987.}

%    General info

\date{}

\dedicatory{This article is dedicated to Professor Akira Yasuhara on his 60th birthday.}


\begin{abstract}
We demonstrate the existence of minimal simplicial $n$-complexes which inevitably contain a nonsplittable two-component link formed by an $(n-1)$-sphere and an $n$-sphere in any embedding into $\mathbb{R}^{2n}$. This provides a higher-dimensional generalization of graphs that are not non-separating planar. 
\end{abstract}

\maketitle

\section{Introduction} 

Throughout this paper, we work in the piecewise linear category. For the fundamentals of piecewise-linear topology, we refer the reader to \cite{H69}, \cite{RS82}. Let $K$ be a finite simplicial $n$-complex, which we identify with its polyhedron in this context. For an integer $k\le n$, let $\varDelta^{k}(K)$ denote the set of all $k$-simplices in $K$. Then the subcomplex $\bigcup_{l\le k}\varDelta^{l}(K)$ of $K$ is called the {\it $k$-skeleton} of $K$ and denoted by $K^{k}$. Let $f$ be an embedding of $K$ into $\mathbb{R}^{m}$. Let $\Lambda^{p,q}(K)$ be the set of all unordered pairs $\{\gamma, \gamma'\}$ of mutually disjoint subcomplexes of $K$ such that $\gamma$ is homeomorphic to the $p$-sphere and $\gamma'$ is homeomorphic to the $q$-sphere. We identify any unordered pair $\{\gamma,\gamma'\}$ in $\Lambda^{p,q}(K)$ with the disjoint union $\gamma\sqcup \gamma'$. Then, for any $\lambda = \gamma \sqcup \gamma' \in \Lambda^{p,q}(K)$, the image $f(\lambda) = f(\gamma) \sqcup f(\gamma')$ forms a two-component link in $f(K)$ consisting of a $p$-sphere and a $q$-sphere. If $p+q = m-1$, the \textit{$\mathbb{Z}_2$-linking number} ${\rm lk}_2(L) = {\rm lk}_2(S^p, S^q) = {\rm lk}_2(S^q, S^p) \in \mathbb{Z}_2$ of a two-component link $L = S^p \sqcup S^q$ in ${\mathbb R}^m$ is well-defined (cf.\ \cite{ST80}). If $m=2n+1$, there are known examples of simplicial $n$-complexes with the following property: every embedding into $\mathbb{R}^{2n+1}$ necessarily contains a two-component link of $n$-spheres with $\mathbb{Z}_2$-linking number $1$. Such examples were constructed by Segal--Spie\.{z} \cite{SeSp92}, Lovasz--Schrijver \cite{LS98}, Skopenkov \cite{Sk03}, Taniyama \cite{T00}, and more recently by the author \cite{N26}. In particular, the case of $n=1$ is famous as the Conway--Gordon--Sachs theorem (\cite{CG83}, \cite{S84}), and such a simplicial $1$-complex (i.e., a graph) is said to be {\it intrinsically linked}.

The purpose of this paper is to present an example of a minimal simplicial $n$-complex that possesses such an intrinsic linking property with respect to embeddings into $\mathbb{R}^{2n}$. Let us introduce three simplicial $n$-complexes $N_{1}^{(n)}$, $N_2^{(n)}$ and $N_3^{(n)}$ as follows. First, let $\sigma_{m}=|a_{0}a_{1}\cdots a_{m}|$ be an $m$-simplex whose vertices are  $a_{0},a_{1},\ldots,a_{m}$. This can be naturally regarded as a simplicial $m$-complex consisting of itself and all its faces. In this setting, the $k$-skeleton of $\sigma_{m}$ is standardly denoted by $\sigma_{m}^{k}$ for $k\le m$. Then we define 
\begin{align*}
N_{1}^{(n)} = \sigma_{2n+2}^{n}\setminus \{|a_{0}a_{j_{1}}\cdots a_{j_{n}}|\ |\ 1\le j_{1}<j_{2}<\cdots<j_{n}\le 2n+2\}. 
\end{align*}
In other words, $N_{1}^{(n)}$ is the subcomplex of $\sigma_{2n+2}^{n}$ obtained by removing all $n$-simplices that contain the vertex $a_0$. Next, for a positive integer $k$, let $[k]^{*m}$ denote the $m$-fold join of $k$ points. Explicitly, we set
\begin{align*}
[k]^{*m} = 
\{a_{0}^{0},a_{1}^{0},\ldots,a_{k-1}^{0}\} * \{a_{0}^{1},a_{1}^{1},\ldots,a_{k-1}^{1}\} * \cdots * \{a_{0}^{n},a_{1}^{n},\ldots,a_{k-1}^{n}\}. 
\end{align*}
Note that $[k]^{*n+1}$ can naturally be regarded as a simplicial $n$-complex. Then we define 
\begin{align*}
N_{2}^{(n)} = [3]^{*n+1}\setminus \{|a_{0}^{0}a_{j_{1}}^{1}\cdots a_{j_{n}}^{n}|\ |\ j_{1},j_{2},\ldots,j_{n}\in\{0,1,2\}\}. 
\end{align*}
In other words, $N_{2}^{(n)}$ is the subcomplex of $[3]^{*n+1}$ obtained by removing all $n$-simplices that contain the vertex $a_{0}^{0}$. Finally, for each integer $l=2,3,\ldots,n+1$, let $F_{l}$ be the family of all $n$-simplices in $\sigma_{2n+2}^{n}$ that are spanned by exactly $l$ vertices from $\{a_0, a_1, \dots, a_{l}\}$ and exactly $n-l+1$ vertices from $\{a_{n+2}, a_{n+3}, \dots, a_{2n+2}\}$. Equivalently, $F_{l}$ consists of all $n$-simplices of the form $|a_{i_{1}} \cdots a_{i_{l}} a_{j_{1}} \cdots a_{j_{n-l+1}}|$ where $0 \leq i_{1} < \cdots < i_{l} \leq l$ and $n+2 \leq j_{1} < \cdots < j_{n-l+1} \leq 2n+2$. Then we define 
\begin{align*}
N_3^{(n)} = \sigma_{2n+2}^{n} \setminus \bigcup_{l=2}^{n+1} F_{l}.
\end{align*}
As we will see in the next section, it is well known that neither $\sigma_{2n+2}^{n}$ nor $[3]^{*n+1}$ can be embedded in ${\mathbb R}^{2n}$ (Theorem \ref{VKF}). On the other hand, it is also known that any proper subcomplex of $\sigma_{2n+2}^{n}$ and of $[3]^{*n+1}$ can be embedded in $\mathbb{R}^{2n}$ (Gr\"{u}nbaum \cite{Grun69}). Therefore, each of $N_1^{(n)}$, $N_2^{(n)}$ and $N_3^{(n)}$ is embeddable in $\mathbb{R}^{2n}$. Then we have the following: 

\begin{Theorem}\label{ILR2n}
Let $n$ be a positive integer, and let $K$ be a simplicial $n$-complex that is $N_1^{(n)}$, $N_2^{(n)}$ or $N_3^{(n)}$. Then, for every embedding $f$ of $K$ into ${\mathbb R}^{2n}$, there exists an element $\lambda$ in $\Lambda^{n-1,n}(K)$ such that ${\rm lk}_{2}(f(\lambda)) = 1$. 
\end{Theorem}

Therefore, $N_1^{(n)}$, $N_2^{(n)}$ and $N_3^{(n)}$ are simplicial $n$-complexes that are `intrinsically linked' with respect to embeddings into ${\mathbb R}^{2n}$. We note that the simplicial $n$-complex $\sigma_{2n+2}^{n}\setminus\{|a_{0}a_{1}\cdots a_{n}|\}$ has already been shown to satisfy this property (Freedman--Krushkal \cite[Proposition 4.1]{FK14}; see also Freedman--Krushkal--Teichner \cite[Lemma 6]{FKT94}). However, both our $N_{1}^{(n)}$ and $N_{3}^{(n)}$ are proper subcomplexes of this complex. Moreover, among all simplicial $n$-complexes possesing this property, the complexes $N_{1}^{(n)}$, $N_{2}^{(n)}$ and $N_{3}^{(n)}$ are minimal in the following sense.

\begin{Proposition}\label{min}
Let $N$ be a proper subcomplex of $N_1^{(n)}$, $N_2^{(n)}$ or $N_3^{(n)}$. Then there exists an embedding $f:N\to {\mathbb R}^{2n}$ such that ${\rm lk}_{2}(f(\lambda))= 0$ for any element $\lambda$ in $\Lambda^{n-1,n}(N)$.
\end{Proposition}

A central role in the proof of Theorem \ref{ILR2n} is played by the van Kampen--Flores theorem. In the next section, we recall this theorem and present proofs of Theorem \ref{ILR2n} and Proposition \ref{min} using it. The van Kampen--Flores theorem is also used to construct examples of simplicial $n$-complexes that are intrinsically linked with respect to embeddings into $\mathbb{R}^{2n+1}$, see Nikkuni \cite{N26}.

\begin{Remark}\label{nsp}
A graph $G$ is said to be {\it non-separating planar} if there exists an embedding $f: G \to \mathbb{R}^2$ such that, for every cycle $\gamma$ in $G$, all vertices of $G$ not lying on $\gamma$ belong to the same connected component of $\mathbb{R}^2 \setminus f(\gamma)$. Equivalently, $G$ is not non-separating planar if and only if for every embedding $f:G\to \mathbb{R}^2$, the image $f(G)$ contains a two-component link consisting of a $0$-sphere and a $1$-sphere with $\mathbb{Z}_2$-linking number $1$. Dehkordi--Farr proved in \cite{DF} that a graph is non-separating planar if and only if it contains no subdivision of $K_1 \sqcup K_4$, $K_1 \sqcup K_{2,3}$, or $K_{1,1,3}$ as illustrated in Fig. \ref{nonsp}. They further applied this characterization to the study of linkless embeddings of graphs in $\mathbb{R}^3$. For related results, see Fowler--Li--Pavelescu \cite{FLP23}, Pavelescu--Pavelescu \cite{PP24}. Note that $N_1^{(1)}$ is isomorphic to $K_1 \sqcup K_4$, $N_2^{(1)}$ is isomorphic to $K_1 \sqcup K_{2,3}$, and $N_3^{(1)}$ is isomorphic to $K_{1,1,3}$. Thus, our simplicial $n$-complexes $N_{1}^{(n)}$, $N_{2}^{(n)}$ and $N_{3}^{(n)}$ may be regarded as higher-dimensional analogues of the minimal subgraphs that prevent a graph from being non-separating planar.
\end{Remark}

\begin{figure}[htbp]
\begin{center}
\scalebox{0.525}{\includegraphics*{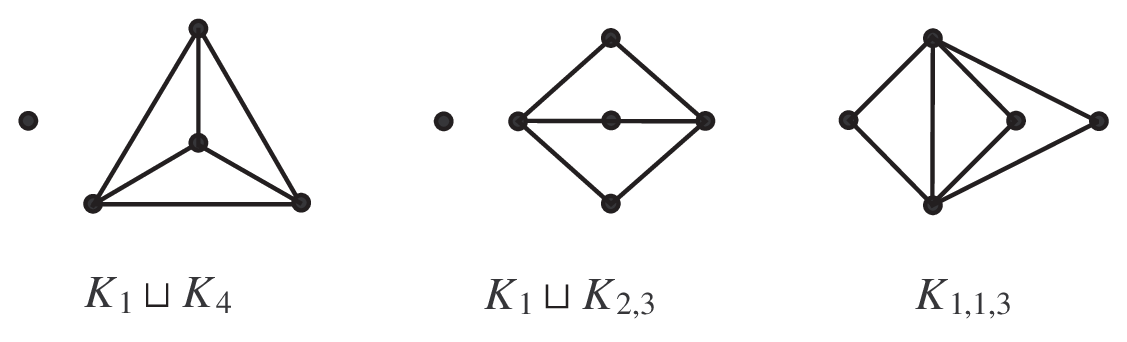}}
%\scalebox{0.65}{\includegraphics[bb = 36 404 577 568, width=15cm]{nonsp.pdf}}
\caption{$K_1 \sqcup K_4$, $K_1 \sqcup K_{2,3}$ and $K_{1,1,3}$}
\label{nonsp}
\end{center}
\end{figure}

\section{Proof of Theorem \ref{ILR2n}} 

First, recall that for any positive integer $n$, neither $\sigma_{2n+2}^{n}$ nor $[3]^{*n+1}$ can be embedded in ${\mathbb R}^{2n}$. It is well-known that every simplicial $n$-complex $K$ can be generically immersed into $\mathbb{R}^{2n}$, where an immersion $\varphi$ of $K$ into $\mathbb{R}^{2n}$ is said to be \textit{generic} if all singularities of $\varphi(K)$ are transversal double points occurring between the interiors of pairs of $n$-simplices. For a generic immersion $\varphi$ of $K$ into $\mathbb{R}^{2n}$ and for any two mutually disjoint $n$-simplices $\sigma, \tau \in \varDelta^n(K)$, let $l(\varphi(\sigma), \varphi(\tau))$ denote the number of double points between $\varphi(\sigma)$ and $\varphi(\tau)$. Then the following fact is known as the \textit{van Kampen--Flores theorem modulo two}.

\begin{Theorem}{\rm (van Kampen \cite{VK33}, Flores \cite{Flores32})}\label{VKF} 
Let $n$ be a positive integer, and let $K$ be a simplicial $n$-complex that is $\sigma_{2n+2}^{n}$ or $[3]^{*n+1}$. Then for every generic immersion $\varphi$ of $K$ into ${\mathbb R}^{2n}$, the following holds: 
\begin{align*}
\sum_{\substack{\{s,s'\}\subset\varDelta^{n}(K)\\ s\cap s'=\emptyset}}l(\varphi(s),\varphi(s'))\equiv 1\pmod{2},
\end{align*}
where the sum runs over all unordered pairs $\{s,s'\}$ of distinct $n$-simplices with $s\cap s'=\emptyset$. 
\end{Theorem}

In Theorem \ref{VKF}, the case of $n=1$ corresponds to the well-known fact that neither $K_{5}$ nor $K_{3,3}$ can be embedded in the plane.

\begin{proof}[Proof of Theorem \ref{ILR2n}]
We first consider the case $K = N_1^{(n)}$. In what follows, for an $m$-simplex $s$, we will also denote its $(m-1)$-skeleton $s^{m-1}$ by $\partial s$. We define collections $\Gamma^{n-1}$ and $\Gamma^n$ of $(n-1)$ and $n$-subcomplexes of $N_1^{(n)}$, respectively, as follows: 
\begin{align*}
\Gamma^{n-1} &= \left\{\partial s \mid s\in \varDelta^{n}(\sigma_{2n+2}^{n}),\,a_{0}\in s\right\},\\
\Gamma^{n} &= \left\{\partial t \mid t\in \varDelta^{n+1}(\sigma_{2n+2}),\,a_{0}\notin t\right\}. 
\end{align*}
Then every $\lambda \in \Lambda^{n-1,n}(N_1^{(n)})$ is the disjoint union of an element $\gamma$ of $\Gamma^{n-1}$ and an element $\gamma'$ of $\Gamma^n$, see Fig. \ref{N_1_link} for $n=2$. Let $f:N_{1}^{(n)}\to {\mathbb R}^{2n}$ be an embedding. This can be extended to a generic immersion $\hat{f}:\sigma_{2n+2}^{n} \to \mathbb{R}^{2n}$ such that, for any two $n$-simplices $s,s' \in \varDelta^n(\sigma_{2n+2}^{n})$ with $a_0 \in s$ and $a_0 \notin s'$, the images $\hat{f}(s)$ and $\hat{f}(s')$ intersect transversely at finitely many double points, all of which occur in the interiors of both $n$-simplices. Note that for any $\lambda = \gamma \sqcup \gamma' \in \Lambda^{n-1,n}(N_1^{(n)})$ with $\gamma = \partial s$, $s \in \varDelta^n(\sigma_{2n+2}^n)$ and $a_0 \in s$, we have 
\begin{align}\label{lkdef}
{\rm lk}_{2}(f(\gamma),f(\gamma'))
\equiv \sum_{s'\in \varDelta^{n}(\gamma')}l(\hat{f}(s),\hat{f}(s'))\pmod{2}.
\end{align}
Note also that for any $\lambda=\partial s\sqcup\gamma'\in\Lambda^{n-1,n}(N_1^{(n)})$, each $n$-simplex $s'\in\varDelta^n(\gamma')$ determines a unique unordered pair $\{s,s'\}$ with $s\cap s'=\emptyset$, and conversely, every such unordered pair $\{s,s'\}\subset\varDelta^n(\sigma_{2n+2}^n)$ with $a_0\in s$ and $a_0\notin s'$ arises from exactly one $\lambda$. Then, by (\ref{lkdef}) and Theorem \ref{VKF}, we obtain 
\begin{align*}\label{lksum}
\sum_{\lambda\in \Lambda^{n-1,n}(N_{1}^{(n)})}{\rm lk}_{2}(f(\lambda))
&= \sum_{\substack{\gamma\sqcup\gamma'\in \Lambda^{n-1,n}(N_{1}^{(n)}) \\ \gamma\in \Gamma^{n-1},\,\gamma'\in\Gamma^{n}}}{\rm lk}_{2}(f(\gamma),f(\gamma'))\\
&\equiv \sum_{\substack{\gamma\sqcup\gamma'\in \Lambda^{n-1,n}(N_{1}^{(n)}) \\ \gamma=\partial s,\,s\in \varDelta^{n}(\sigma_{2n+2}^{n}),\,a_{0}\in s}}
\bigg(\sum_{s'\in \varDelta^{n}(\gamma')}l(\hat{f}(s),\hat{f}(s'))\bigg)\\
&= \sum_{\substack{\{s,\,s'\}\subset \varDelta^{n}(\sigma_{2n+2}^{n}) \\ s\cap s'=\emptyset}} 
l(\hat{f}(s),\hat{f}(s'))\\
&\equiv 1\pmod{2}. 
\end{align*}
Thus, there exists some $\lambda$ in $\Lambda^{n-1,n}(N_{1}^{(n)})$ such that ${\rm lk}_{2}(f(\lambda))=1$.

Next, we consider the case $K = N_{2}^{(n)}$. We define collections $\Gamma^{n-1}$ and $\Gamma^n$ of $(n-1)$ and $n$-subcomplexes of $N_{2}^{(n)}$, respectively, as follows: 
\begin{align*}
\Gamma^{n-1} &= \left\{\partial s \mid s\in \varDelta^{n}([3]^{*n+1}),\,a_{0}^{0}\in s\right\},\\
\Gamma^{n} &= \left\{\left\{a_{1}^{0},a_{2}^{0}\right\}*\left\{a_{i_{1}}^{1},a_{j_{1}}^{1}\right\}*\cdots*\left\{a_{i_{n}}^{n},a_{j_{n}}^{n}\right\} \mid 0\le i_{q}<j_{q}\le 2,\,q\ge 1\right\}. 
\end{align*}
Note that the $(n+1)$-fold join of two points is homeomorphic to the $n$-sphere. Then every $\lambda \in \Lambda^{n-1,n}(N_{2}^{(n)})$ is the disjoint union of an element $\gamma$ of $\Gamma^{n-1}$ and an element $\gamma'$ of $\Gamma^n$, see Fig. \ref{N_2_link} for $n=2$. Let $f:N_{2}^{(n)}\to {\mathbb R}^{2n}$ be an embedding. This can be extended to a generic immersion $\hat{f}:[3]^{*n+1} \to \mathbb{R}^{2n}$ such that, for any two $n$-simplices $s,s' \in \varDelta^n([3]^{*n+1})$ with $a_{0}^{0} \in s$ and $a_{0}^{0} \notin s'$, the images $\hat{f}(s)$ and $\hat{f}(s')$ intersect transversely at finitely many double points, all of which occur in the interiors of both $n$-simplices. Then, in a similar way as in the case of $N_{1}^{(n)}$, we have 
\begin{align*}%\label{lksum2}
\sum_{\lambda\in \Lambda^{n-1,n}(N_{2}^{(n)})}{\rm lk}_{2}(f(\lambda))
&= \sum_{\substack{\gamma\sqcup\gamma'\in \Lambda^{n-1,n}(N_{2}^{(n)}) \\ \gamma\in \Gamma^{n-1},\,\gamma'\in\Gamma^{n}}}{\rm lk}_{2}(f(\gamma),f(\gamma'))\\
&\equiv \sum_{\substack{\gamma\sqcup\gamma'\in \Lambda^{n-1,n}(N_{2}^{(n)}) \\ \gamma=\partial s,\,s\in \varDelta^{n}([3]^{*n+1}),\,a_{0}^{0}\in s}}
\bigg(\sum_{s'\in \varDelta^{n}(\gamma')}l(\hat{f}(s),\hat{f}(s'))\bigg)\\
&= \sum_{\substack{\{s,\,s'\}\subset \varDelta^{n}([3]^{*n+1}) \\ s\cap s'=\emptyset}} 
l(\hat{f}(s),\hat{f}(s'))\\
&\equiv 1\pmod{2}. 
\end{align*}
Thus the result follows.

Finally, we consider the case $K = N_{3}^{(n)}$. We define collections $\Gamma^{n-1}$ and $\Gamma^n$ of $(n-1)$ and $n$-subcomplexes of $N_{3}^{(n)}$, respectively, as follows:
\begin{align*}
\Gamma^{n-1} &= \Big\{\partial s \;\Big|\; s\in \bigcup_{l=2}^{n+1}F_{l}\Big\},\\
\Gamma^{n} &= \Big\{\partial t \;\Big|\; t\in \varDelta^{n+1}(\sigma_{2n+2}),\,\varDelta^{n}(t)\cap\Big(\bigcup_{l=2}^{n+1}F_{l}\Big)=\emptyset\Big\}. 
\end{align*}
Note that any two $n$-simplices in $\bigcup_{l=2}^{n+1} F_{l}$ share at least one common vertex. Now, for $\gamma = \partial s \in \Gamma^{n-1}$, let $t$ be the $(n+1)$-simplex spanned by the $n+2$ vertices not contained in $s$. Then every $s' \in \varDelta^{n}(t)$ is disjoint from $s$ and thus does not belong to $\bigcup_{l=2}^{n+1} F_l$. Consequently, $\gamma' = \partial t$ belongs to $\Gamma^n$. Then every $\lambda \in \Lambda^{n-1,n}(N_3^{(n)})$ is the disjoint union of an element $\gamma$ of $\Gamma^{n-1}$ and an element $\gamma'$ of $\Gamma^n$, see Fig. \ref{N_3_link} for $n=2$. Let $f:N_{3}^{(n)}\to {\mathbb R}^{2n}$ be an embedding. This can be extended to a generic immersion $\hat{f}:\sigma_{2n+2}^{n} \to \mathbb{R}^{2n}$ such that, for any two $n$-simplices $s,s' \in \varDelta^n(\sigma_{2n+2}^{n})$ with $s\in \bigcup_{l=2}^{n+1}F_{l}$ and $s'\in \varDelta^{n}(\sigma_{2n+2})\setminus \bigcup_{l=2}^{n+1}F_{l}$, the images $\hat{f}(s)$ and $\hat{f}(s')$ intersect transversely at finitely many double points, all of which occur in the interiors of both $n$-simplices. Note that for any $\lambda=\partial s\sqcup\gamma'\in\Lambda^{n-1,n}(N_{3}^{(n)})$, each $n$-simplex $s'\in\varDelta^n(\gamma')$ determines a unique unordered pair $\{s,s'\}$ with $s\cap s'=\emptyset$, and conversely, every such unordered pair $\{s,s'\}\subset\varDelta^n(\sigma_{2n+2}^n)$ with $s\in \bigcup_{l=2}^{n+1}F_{l}$ and $s'\in \varDelta^{n}(\sigma_{2n+2})\setminus \bigcup_{l=2}^{n+1}F_{l}$ arises from exactly one $\lambda$. Then, in the same way as in the case of $N_{1}^{(n)}$, we obtain 
\begin{align*}
\sum_{\lambda\in \Lambda^{n-1,n}(N_{3}^{(n)})}{\rm lk}_{2}(f(\lambda))
&= \sum_{\substack{\gamma\sqcup\gamma'\in \Lambda^{n-1,n}(N_{3}^{(n)}) \\ \gamma\in \Gamma^{n-1},\,\gamma'\in\Gamma^{n}}}{\rm lk}_{2}(f(\gamma),f(\gamma'))\\
&\equiv \sum_{\substack{\gamma\sqcup\gamma'\in \Lambda^{n-1,n}(N_{3}^{(n)}) \\ \gamma=\partial s,\,s\in \bigcup_{l=2}^{n+1}F_{l}}}
\bigg(\sum_{s'\in \varDelta^{n}(\gamma')}l(\hat{f}(s),\hat{f}(s'))\bigg)\\
&= \sum_{\substack{\{s,\,s'\}\subset \varDelta^{n}(\sigma_{2n+2}^{n}) \\ s\cap s'=\emptyset}} 
l(\hat{f}(s),\hat{f}(s'))\\
&\equiv 1\pmod{2}. 
\end{align*}
Thus the result follows. This completes the proof. 
\end{proof}

\begin{figure}[htbp]
\begin{center}
\scalebox{0.575}{\includegraphics*{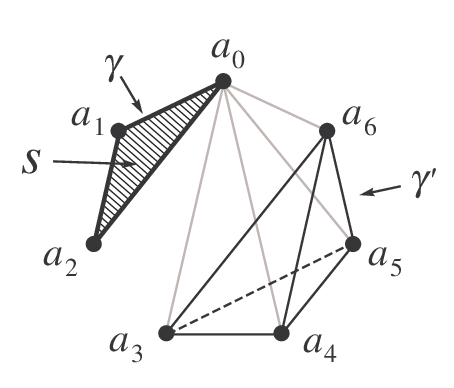}}
%\scalebox{0.65}{\includegraphics[bb = 206 302 424 488, width=7cm]{N_1_link.pdf}}
\caption{$\lambda=\gamma\sqcup \gamma'\in \Lambda^{1,2}(N_{1}^{(2)})$}
\label{N_1_link}
\end{center}
\end{figure}

\begin{figure}[htbp]
\begin{center}
\scalebox{0.6}{\includegraphics*{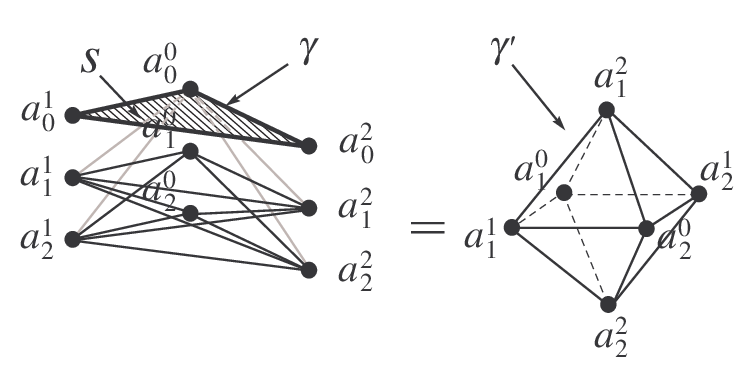}}
%\scalebox{0.75}{\includegraphics[bb = 125 304 487 488, width=10cm]{N_2_link.pdf}}
\caption{$\lambda=\gamma\sqcup \gamma'\in \Lambda^{1,2}(N_{2}^{(2)})$}
\label{N_2_link}
\end{center}
\end{figure}

\begin{figure}[htbp]
\begin{center}
\scalebox{0.575}{\includegraphics*{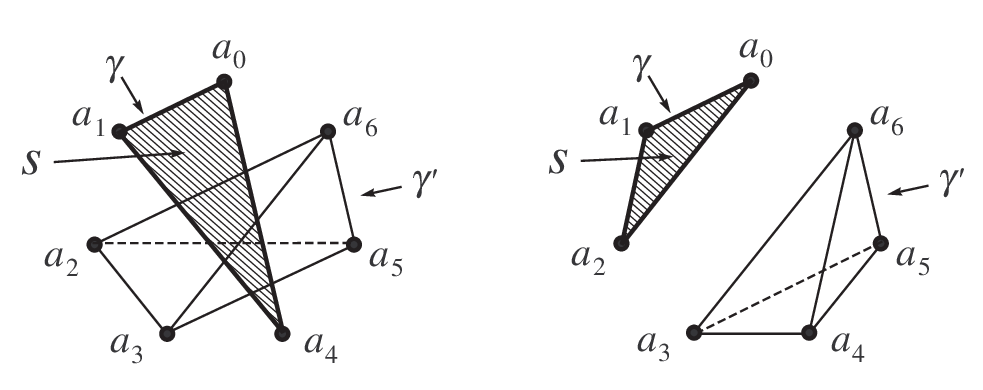}}

%\scalebox{0.65}{\includegraphics[bb = 206 302 424 488, width=7cm]{N_3_link.pdf}}
\caption{$\lambda=\gamma\sqcup \gamma'\in \Lambda^{1,2}(N_3^{(2)})$}
\label{N_3_link}
\end{center}
\end{figure}

\begin{proof}[Proof of Proposition \ref{min}]
Let $H$ be either $\sigma_{2n+2}^{n}$ or $[3]^{*n+1}$. It is well known that there exists a generic immersion $h : H \to \mathbb{R}^{2n}$ with exactly one double point, which occurs between the interiors of a pair of disjoint $n$-simplices (van Kampen \cite{VK28}). By the symmetry of $H$, we may assume without loss of generality that this double point occurs between $h(|a_{0}a_{1}\cdots a_{n}|)$ and $h(|a_{n+1}a_{n+2}\cdots a_{2n+1}|)$ if $H=\sigma_{2n+2}^{n}$, and between $h(|a_{0}^{0}a_{1}^{0}\cdots a_{n}^{0}|)$ and $h(|a_{0}^{1}a_{1}^{1}\cdots a_{n}^{1}|)$ if $H=[3]^{*n+1}$. 
We first consider the cases of $N_1^{(n)}$ and $N_2^{(n)}$. For each $i=1,2$, let $h_i : N_i^{(n)} \to \mathbb{R}^{2n}$ be the restriction of $h$ to the subcomplex $N_i^{(n)} \subset H$. Each $h_{i}$ is an embedding, and the image $h(N_{i}^{(n)})$ contains exactly one two-component link $L_{i}$ consisting of an $(n-1)$-sphere and an $n$-sphere with ${\mathbb Z}_{2}$-linking number $1$. These links are given explicitly as follows:

$\bullet$ for $i=1$: $L_{1} = h(\partial |a_{0}a_{1}\ldots a_{n}|)\sqcup h(\partial |a_{n+1}a_{n+2}\cdots a_{2n+2}|)$,  

$\bullet$ for $i=2$: $L_{2} = h(\partial |a_{0}^{0}a_{1}^{0}\cdots a_{n}^{0}|)\sqcup h(\{a_{1}^{0},a_{2}^{0}\}*\{a_{1}^{1},a_{2}^{1}\}*\cdots *\{a_{1}^{n},a_{2}^{n}\})$.

We now show that if $N$ is any proper subcomplex of $N_{i}^{(n)}$ ($i=1,2$), then $h(N)$ does not contain any two-component link consisting of an $(n-1)$-sphere and an $n$-sphere with $\mathbb{Z}_2$-linking number $1$. To see this, it is sufficient to check the following cases: 

\begin{enumerate}
\item For $N_{1}^{(n)}$, every maximal proper subcomplex obtained by removing a single simplex is, up to relabeling of the vertices, isomorphic to either $N_{1}^{(n)}\setminus \{|a_{0}a_{1}\cdots a_{n-1}|\}$ or $N_{1}^{(n)}\setminus \{|a_{n+1}a_{n+2}\cdots a_{2n+1}|\}$. Thus, it suffices to consider $N=N_{1}^{(n)}\setminus \{|a_{0}a_{1}\cdots a_{n-1}|\}$ or $N_{1}^{(n)}\setminus \{|a_{n+1}a_{n+2}\cdots a_{2n+1}|\}$. In both cases, $h(N)$ does not contain $L_{1}$.

\item For $N_{2}^{(n)}$, every maximal proper subcomplex obtained by removing a single simplex is, up to relabeling of the vertices, isomorphic to either $N_{2}^{(n)}\setminus \{|a_{0}^{0}a_{0}^{1}\cdots a_{0}^{n-1}|\}$ or $N_{2}^{(n)}\setminus \{|a_{1}^{0}a_{1}^{1}\cdots a_{1}^{n}|\}$. Thus, it suffices to consider $N=N_{2}^{(n)}\setminus \{|a_{0}^{0}a_{0}^{1}\cdots a_{0}^{n-1}|\}$ or $N_{2}^{(n)}\setminus \{|a_{1}^{0}a_{1}^{1}\cdots a_{1}^{n}|\}$. In both cases, $h(N)$ does not contain $L_{2}$. 
\end{enumerate}

Therefore, in all cases, no such link appears in $h(N)$ for any proper subcomplex $N$ of $N_i^{(n)}\ (i=1,2)$. 

Next, we consider the case of $N_3^{(n)}$. For each $l = 2, 3, \dots, n+1$, let $\gamma = \partial s \in \Gamma^{n-1}$, where
\begin{eqnarray*}
s = |a_0 a_1 \cdots a_{l-1} a_{n+2} a_{n+3} \cdots a_{2n-l+2}|.
\end{eqnarray*}
Then the remaining vertices form an element $\gamma' \in \Gamma^n$ given by
\begin{eqnarray*}
\gamma' = \partial |a_l a_{l+1} \cdots a_{n+1} a_{2n-l+3} a_{2n-l+4}\cdots a_{2n+2}|.
\end{eqnarray*} 
By relabeling the vertices if necessary and restricting the generic immersion $h:\sigma_{2n+2}^{(n)}\to \mathbb{R}^{2n}$ to $N_3^{(n)}$, 
we obtain an embedding $h_l : N_3^{(n)} \to \mathbb{R}^{2n}$ 
such that $h_l(N_3^{(n)})$ contains exactly one two-component link $L_3 = h_l(\gamma) \sqcup h_l(\gamma')$ with $\mathbb{Z}_2$-linking number $1$. By the symmetry of $N_3^{(n)}$, any subcomplex $N$ obtained by removing exactly one $n$-simplex from $N_3^{(n)}$ is isomorphic to a subcomplex obtained by removing exactly one $n$-simplex from $\gamma'$ for a suitable choice of $l$. For the corresponding embedding $h_l$, the image $h_l(N)$ does not contain the link $L_3$. Thus, every proper subcomplex of $N_3^{(n)}$ admits an embedding into $\mathbb{R}^{2n}$ in which all two-component links have $\mathbb{Z}_2$-linking number $0$. This completes the proof. 
\end{proof}

As noted in Remark \ref{nsp}, for $n=1$, the only minimal simplicial $1$-complexes with this intrinsic linking property are $N_{1}^{(1)}$, $N_{2}^{(1)}$ and $N_{3}^{(1)}$. By contrast, the author suspects that further examples exist for $n\ge 2$. Therefore, let us pose the following problem.

\begin{Problem}\label{q1}
List all minimal simplicial $n$-complexes that possess the intrinsic linking property in the case $n \ge 2$. 
\end{Problem}

\begin{Remark}\label{4di}
In particular, for $n=2$, it is known that there exists a nonsplittable link consisting of two $2$-spheres in ${\mathbb R}^{4}$ (van Kampen \cite{VK28}, Zeeman \cite{Z65}, see also \cite[pp. 88--95]{Rolfsen76}). The ${\mathbb Z}_{2}$-linking number in the above sense is not defined for such links, and it is not so easy to prove that they are nonsplittable. For a simplicial $2$-complex $K$ to contain a disjoint union of two $2$-spheres, at least eight $0$-simplices are required. As mentioned in Section 1, actually any embedding of $\sigma_{7}^{2}$ into ${\mathbb R}^{5}$ contains a two-component link of $2$-spheres whose ${\mathbb Z}_{2}$-linking number is $1$, but $\sigma_{7}^{2}$ cannot be embedded into ${\mathbb R}^{4}$. The author is not certain, but it seems doubtful that there exists a simplicial $2$-complex whose every embedding into ${\mathbb R}^4$ contains a nonsplittable two-component link of $2$-spheres.
\end{Remark}

\section*{Acknowledgment}

The author was supported by JSPS KAKENHI Grant Number 25K06987.

{\normalsize
}

\end{document}